\documentclass[a4paper,reqno]{amsart}
\usepackage{amsmath}
\usepackage{amssymb}
\usepackage{amsthm}
\usepackage{latexsym}
\usepackage{amscd}
\usepackage{xypic}
\xyoption{curve}
\usepackage{ifthen} 
\usepackage{hyperref} 
\usepackage{graphicx}
\usepackage{enumerate}
\usepackage{enumitem}
\usepackage{rotating}
\usepackage{xcolor}
\usepackage{mathtools}
\usepackage{mathrsfs} 
\usepackage{relsize}
\usepackage{layout}
\usepackage{todonotes}

\numberwithin{equation}{section}

\def\cocoa{{\hbox{\rm C\kern-.13em o\kern-.07em C\kern-.13em o\kern-.15em A}}}

\newtheorem{theorem}{Theorem}[section]

\newtheorem{lemma}[theorem]{Lemma}
\newtheorem{proposition}[theorem]{Proposition}
\newtheorem{corollary}[theorem]{Corollary}

\theoremstyle{definition}
\newtheorem{remark}[theorem]{Remark}
\newtheorem{definition}[theorem]{Definition}
\newtheorem{example}[theorem]{Example}

\newtheorem{construction}[theorem]{Construction}

\newcommand {\sHom}{\mathcal{H}\kern -0.25ex{\mathit om}}
\newcommand {\sExt}{\mathcal{E}\kern -0.25ex{\mathit xt}}
\newcommand {\sTor}{\mathcal{T}\kern -0.25ex{\mathit or}}
\newcommand {\Spl}{Spl}
\newcommand {\im}{\mathrm{im}}
\newcommand {\rk}{\mathrm{rk}}

\newcommand {\Ext}{\mathrm{Ext}}
\newcommand {\Hom}{\mathrm{Hom}}
\newcommand {\Hilb}{\mathcal{H}\kern -0.25ex{\mathit ilb\/}}

\newcommand {\rH}{{H}}
\newcommand {\rh}{{h}}

\newcommand {\cK}{\mathcal{K}}

\newcommand {\bC}{\mathbb{C}}
\newcommand {\bP}{\mathbb{P}}

\newcommand{\cE}{{\mathcal E}}
\newcommand{\cF}{{\mathcal F}}
\newcommand{\cM}{{\mathcal M}}

\newcommand{\cO}{{\mathcal O}}
\newcommand{\cG}{{\mathcal G}}
\newcommand{\cI}{{\mathcal I}}

\newcommand{\Pic}{\operatorname{Pic}}

\newcommand{\NS}{\operatorname{NS}}

\def\p#1{{\bP^{#1}}}

\def\ga#1{{{\accent"12 #1}}}
 
\def\mapright#1{\mathbin{\smash{\mathop{\longrightarrow}
\limits^{#1}}}}

\title[Special Ulrich bundles on regular surfaces]{Special Ulrich bundles on regular surfaces\\ with non--negative Kodaira dimension}

\thanks{}



\author[Gianfranco Casnati]{Gianfranco Casnati}
\thanks{The author is a member of GNSAGA group of INdAM and is supported by the framework of the MIUR grant Dipartimenti di Eccellenza 2018-2022 (E11G18000350001).}

\begin{document}

\begin{abstract}
Let $S$ be a regular surface endowed with a very ample line bundle $\mathcal O_S(h_S)$. Taking inspiration from a very recent result by D. Faenzi on $K3$ surfaces, we prove that  if $\cO_S(h_S)$ satisfies a short list of technical conditions, then such a polarized surface  supports special Ulrich bundles of rank $2$. As applications, we deal with general embeddings of regular surfaces, pluricanonically embedded regular surfaces and  some properly elliptic surfaces of low degree in $\p N$.
Finally, we also discuss about the size of the families of Ulrich bundles on $S$ and we inspect the existence of special Ulrich bundles on surfaces of low degree. 

\end{abstract}

\maketitle

\section{Introduction}
Let $\p N$ be the projective space of dimension $N$ over an algebraically closed field $k$ of characteristic $0$. If $X\subseteq\p N$ is a {\sl variety}, i.e. an integral closed subscheme, then it is naturally endowed with the very ample line bundle $\cO_X(h_X):=\cO_{\p N}(1)\otimes\cO_X$. We say that a sheaf $\cE$ on $X$ is {\sl Ulrich (with respect to $\cO_X(h_X)$)} if 
$$
h^i\big(X,\cE(-ih_X)\big)=h^j\big(X,\cE(-(j+1)h_X)\big)=0,
$$
for each $i>0$ and $j<\dim(X)$.

Ulrich bundles on a variety $X$ have many properties: we refer the interested reader to \cite{E--S--W}, where the authors also raised the following questions.
\medbreak
\noindent
{\bf Questions.} 
Is every variety (or even scheme) $X\subseteq\p N$ the support of an Ulrich sheaf? If so, what is the smallest possible rank for such a sheaf?
\medbreak

When $C$ is a {\sl curve}, i.e. a smooth variety of dimension $1$, the above questions have very easy answers: indeed if $g$ is the genus of $C$ and $\mathcal L\in\Pic^{g-1}(C)$ satisfies $h^0\big(C,\mathcal L\big)=0$, then $\mathcal L(h_C)$ is an Ulrich line bundle. 

At present, no general answers to the above questions are known when $X$ has dimension greater than $1$, though a great number of partial results have been proved: without any claim of completeness, we recall \cite{A--F--O, Bea2,Cs4,Cs4erratum,C--K--M,E--He, Fa,F--M, F--PL,MR, K--MR1,K--MR2,MR--PL2}. The interested reader can also refer to the recent survey  \cite{Bea3} for further results.

In this paper we study the case of {\sl surfaces}, i.e. smooth varieties of dimension $2$. In particular, following the argument used by D. Faenzi for $K3$ surfaces in \cite{Fa} we deal with surfaces $S$ with $q(S)=0$, partially extending analogous results proved in \cite{Bea2,Cs4,Cs4erratum} when $q(S)=p_g(S)=0$. 

In order to state our main result, we quickly recall a few facts and definitions. Recall that a coherent sheaf $\cG$ on $S$ is called {\sl simple} if $\Hom_S\big(\cG,\cG\big)\cong k$. There exists a (possibly non--separated) coarse moduli space $\Spl_S(r;c_1,c_2)$ parameterizing simple coherent torsion--free sheaves on $S$ with fixed rank $r$ and Chern classes $c_1$, $c_2$ (see \cite{A--K}). 

Moreover, if we set $
\mu(\cG):=c_1(\cG)h_S/\rk(\cG)$, the coherent torsion--free sheaf $\cG$ is {\sl $\mu$--stable (with respect to $\cO_S(h_S)$)} if  $\mu(\mathcal K) < \mu(\cG)$ for each subsheaf $\mathcal K$ with $0<\rk(\mathcal K)<\rk(\cG)$. There exists a quasi--projective scheme $M_S(r;c_1,c_2)$ parameterizing $\mu$--stable coherent torsion--free sheaves on $S$ with fixed rank $r$ and Chern classes $c_1$, $c_2$. Such a $M_S(r;c_1,c_2)$ can be identified with an open subset of $\Spl_S(r;c_1,c_2)$.

An Ulrich bundle $\cE$ of rank $2$ on a surface $S$ is called  {\sl special} if 
$$
c_1(\cE)=c_1^{sp}:=3h_S+K_S,
$$
$K_S$ being a canonical divisor. Special Ulrich bundles are easier to construct than arbitrary Ulrich bundles. Indeed, the Serre duality implies that a rank $2$ bundle $\cE$ with $c_1(\cE)=c_1^{sp}$ is Ulrich if and only if 
$$
h^0\big(S,\cE(-h_S)\big)=h^1\big(S,\cE(-h_S)\big)=0,
$$
Moreover, when $\cE$ is a special Ulrich bundle, it is easy to check that 
$$
c_2(\cE)=c_2^{sp}:=\frac{5h_S^2+3h_SK_S}2+2\chi(\cO_S).
$$
Finally recall that a line bundle $\mathcal L$ is {\sl non--special} if $h^1\big(S,\mathcal L\big)=0$.

We are now able to state the main result of the paper.

\begin{theorem}
\label{tMain}
Let $S$ be a surface  with $q(S)=0$ and $p_g(S)\ge1$, endowed with a very ample non--special line bundle $\cO_S(h_S)$. Assume $h^0\big(S,\cO_S(2K_S-h_S)\big)=0$.

If $h_S^2+4\ge h_SK_S$, then $S$ supports special simple Ulrich bundles $\cE$ of rank $2$. If $h_S^2>h_SK_S$ and the complement $S_0$ of the union of smooth rational curves is dense in $S$, then the aforementioned bundle $\cE$ can be chosen $\mu$--stable.

The point corresponding to $\cE$ in either $\Spl_S(2;c_1^{sp},c_2^{sp})$ or $M_S(2;c_1^{sp},c_2^{sp})$ satisfies  $\dim\Ext_S^2\big(\cE,\cE\big)=p_g(S)$, is smooth and lies in a unique component of dimension $h_S^2-K_S^2+5\chi(S)$.

\end{theorem}

Before dealing with the structure of the paper we make some comments about the hypothesis in the statement above. 

If $p_g(S)=0$ the existence of special Ulrich bundles on $S$ has been proved with the same construction in \cite{Cs4} without the hypothesis $h^0\big(S,\cO_S(2K_S-h_S)\big)=0$. 

Nevertheless, the proof of the $\mu$--stability of $\cE$ given here cannot be extended to the case $p_g(S)=0$ when the Kodaira dimension $\kappa(S)$ is negative, because it rests on the hypothesis $S_0\ne\emptyset$, which necessarily imply $\kappa(S)\ge0$. 

The $\mu$--stability of $\cE$ when $\kappa(S)=-\infty$ has been proved in \cite{Cs4erratum} with a different argument, under the additional hypothesis that $k$ is uncountable. 

A priori, it could happen that $S_0$ is not open even when $\kappa(S)\ge0$. E.g., in \cite{B--MK} the authors give an example of a $K3$ surface containing infinitely many smooth rational curves. 

In particular, we cannot deduce the $\mu$--stability of $\cE$ without further information. E.g. if  $S$ is a surface of general type, i.e. $\kappa(S)=2$, then $S$ contains at most a finite number of smooth rational curves by \cite{L--M}, hence $S_0$ is certainly dense in this case.
 
 Finally, the vanishing $h^0\big(S,\cO_S(2K_S-h_S)\big)=0$ is certainly fullfilled if $\kappa(S)\le1$, thanks to \cite[Corollary 2.2.7]{Laz}. If $\kappa(S)=2$ this is no longer true (see e.g. Example \ref{eRa}).

 \medbreak

We now deal with the content of the paper. In Section \ref{sMain} we prove the above theorem via a finite induction. We first construct a vector bundle of rank $2$ such that $c_1(\cF)=h_S+K_S$, $h^1\big(S,\cF\big)=0$ and $h^0\big(S,\cF\big)=p_g(S)$ using standard techniques (see Section \ref{sStart}). 

Then, via a classical result due to Artamkin (see \cite{Ar}), we construct inductively a sequence of vector bundles  $\cF_d$ of rank $2$ with $1\le d\le p_g(S)$ such that $c_1(\cF_d)=c_1(\cF)$, $h^1\big(S,\cF_d\big)=0$ and $h^0\big(S,\cF_d\big)=p_g(S)-d$ (see Section \ref{sStep}). 

As pointed out above, it follows that $\cE:=\cF_{p_g(S)}(h_S)$ is a special Ulrich bundle of rank $2$: in the paper we show that it has all the properties listed in the statement of Theorem \ref{tMain}. 

It is natural  to ask whether surfaces as in Theorem \ref{tMain} actually exist. 
On the one hand, each surface $S$ with $q(S)=0$ can be endowed with many very ample line bundles satisfying the hypothesis of Theorem \ref{tMain} (see Example \ref{eLarge}). 

On the other hand, there are several interesting polarized surfaces satisfying the hypothesis of Theorem \ref{tMain} besides the case of $K3$ surfaces described in \cite{Fa}. 
For instance, each regular surface of general type $S\subseteq \p N$ with $\cO_S(h_S)\cong\cO_S(n K_S)$ for an $n\ge3$ (see Example \ref{eGeneral}). Some properly elliptic surfaces  in $\p 4$ provide a second interesting example (see Example \ref{eElliptic}). 

In Section \ref{sWild} we make some comments about the size of families of Ulrich bundles on the surfaces described in Examples \ref{eLarge}, \ref{eGeneral} and \ref{eElliptic}.

Finally, in Section \ref{sLow}, we deal with surfaces of low degree in $\p N$, extending some results from \cite{Cs5}.

\subsection{Acknowledgement}
I would like to express my thanks to D. Faenzi and A.F. Lopez for many helpful discussions and suggestions about the content of the present paper. I am particularly indebted with the referee for her/his criticisms, questions, remarks and suggestions which have considerably improved the whole exposition.

\section{Some preliminary facts}
\label{sGeneral}
In this section we list some results which will be used in the paper. For all the other necessary results we refer the reader to \cite{Ha2}.

If $\cG$ and $\mathcal H$ are coherent sheaves on the surface $S$, then  the Serre duality holds
\begin{equation}
\label{Serre}
\Ext_S^i\big(\mathcal H,\cG(K_S)\big)\cong \Ext_S^{2-i}\big(\cG,\mathcal H\big)^\vee,
\end{equation}
(see \cite[Proposition 7.4]{Ha3}: see also \cite{Ba}).
Thus $q(S):=h^1\big(S,\cO_S\big)=h^1\big(S,\cO_S(K_S)\big)$, $p_g(S):=h^2\big(S,\cO_S\big)=h^0\big(S,\cO_S(K_S)\big)$ and $\rh^2\big(S,\cO_S(K_S)\big)=\rh^0\big(S,\cO_S\big)=1$.

Assume that $\mathcal H$ is torsion--free. Then there is a natural injective morphism $\cO_S(K_S)\to\sHom_S\big(\mathcal H,\mathcal H(K_S)\big)$. Thus the map induced on cohomology
$$
\rH^0\big(S,\cO_S(K_S)\big)\to\Hom_S\big(\mathcal H,\mathcal H(K_S)\big)\cong\Ext_S^2\big(\mathcal H,\mathcal H\big)^\vee
$$
is injective too. In particular
\begin{equation}
\label{Mukai}
\dim\Hom_S\big(\mathcal H,\mathcal H(K_S)\big)=\dim\Ext_S^2\big(\mathcal H,\mathcal H\big)\ge p_g(S).
\end{equation}
If $\cG$ is a vector bundle on $S$, then 
the Riemann--Roch theorem for $\cG$ on $S$ is
\begin{equation}
\label{RRGeneral}
\rh^0\big(S,\cG\big)+\rh^{2}\big(S,\cG\big)=\rh^{1}\big(S,\cG\big)+\rk(\cG)\chi(\cO_S)+\frac{c_1(\cG)(c_1(\cG)-K_S)}2-c_2(\cG).
\end{equation}

We finally recall the Cayley--Bacharach construction of vector bundles on $S$. Let $Z\subseteq S$ be a $0$--dimensional locally complete intersection  subscheme and let $\mathcal L\in\Pic(S)$. Recall that  $Z$ satisfies the  {\sl Cayley--Bacharach condition with respect to $\mathcal L$} if
$$
\rh^0\big(S,\cI_{Z\vert S}\otimes\mathcal L\big)=\rh^0\big(S,\cI_{Z'\vert S}\otimes\mathcal L\big)
$$
for each subscheme $Z'\in Z$  with $\deg(Z')=\deg(Z)-1$.

\begin{theorem}
\label{tCB}
Let $S$ be a surface and let $Z\subseteq S$ be a $0$--dimensional locally complete intersection subscheme.

Then there exists a vector bundle $\cF$ of rank $2$ fitting into an exact sequence of the form
$$
0\longrightarrow\cO_S\longrightarrow\cF\longrightarrow\cI_{Z\vert S}\otimes\mathcal L\longrightarrow0
$$
if and only if $Z$ has the Cayley--Bacharach property with respect to $\mathcal L(K_S)$.
\end{theorem}
\begin{proof}
See \cite[Theorem 5.1.1]{H--L}.
\end{proof}

We close this section with a remark about the locus of smooth rational curves on a surface $S$.

\begin{definition}
If $S$ is a surface, we denote by $\mathcal D$ the set of smooth rational curves $D\subseteq S$ and we set $S_0:=S\setminus\bigcup_{D\in\mathcal D} D$.
\end{definition}

\begin{remark}
\label{rS_0}
Notice that if $\kappa(S)=-\infty$, then $S_0=\emptyset$. On the opposite side, in the case we are interested in, i.e. $\kappa(S)\ge0$, then $S_0$ is the complement of a countable union of proper closed subschemes.

To prove this assertion we first notice that $D^2\le-1$ for each $D\in\mathcal D$ when  $\kappa(S)\ge0$ by \cite[Proposition III.2.3]{B--H--P--V}. If $A\in \vert D+\vartheta\vert$ for some $\cO_S(\vartheta)\in \Pic^0(S)$, then $DA\le-1$, hence $D$ is a component of $A$. It follows that $D=A$, because they have the same degree with respect to any very ample line bundle on $S$, i.e. the class of each $D\in\mathcal D$ in the N\'eron--Severi group $\NS(S)$ contains exactly one smooth rational curve. 

We deduce from the above discussion that $S_0$ is certainly non--empty and dense if  $k$ is uncountable (see \cite[Exercise V.4.15 (c)]{Ha2}). 
Nevertheless, as pointed out in the introduction, in many cases $S_0$ is actually open and non--empty, without any restriction on the cardinality of $k$ (see \cite{L--M}). 
\end{remark}

\section{The base case}
\label{sStart}
As explained in the introduction, the proof of Theorem \ref{tMain} is by induction. In this section we deal with its base case.

Let $S$ be a surface with $q(S)=0$ and $\cO_S(h_S)$ a non--special very ample  line bundle. If $h^0\big(S,\cO_S(K_S-h_S)\big)=0$, then Equality \eqref{RRGeneral} yields $h^0\big(S,\cO_S(h_S)\big)=N+1$ where
\begin{equation}
\label{Dimension}
N:=\frac{h^2_S-h_SK_S}2+p_g(S).
\end{equation}
In particular $\cO_S(h_S)$ induces an embedding $S\subseteq\p N$. 

Let $X$ be any scheme. In what follows $X^{[N+2]}$ denotes the Hilbert scheme of $0$--dimensional subschemes of degree $N+2$ inside $X$.

\begin{construction}
\label{conE}
Let $S$ be a surface with  $q(S)=0$, endowed with a non--special very ample line bundle $\cO_S(h_S)$. Let $S\subseteq\p N$ be the induced embedding. 

Since $S$ is integral and non--degenerate in $\p N$, it follows the existence of an open non--empty subset $\mathcal Z\subseteq S^{[N+2]}$ whose points correspond to schemes $Z$  of $N+2$ points in general linear position inside $\p N$. Each scheme $Z$ corresponding to a point in $\mathcal Z$ satisfies the hypothesis of Theorem \ref{tCB} with respect to $\cO_S(h_S)$, hence there is a rank $2$ vector bundle $\cF$ fitting into
\begin{equation}
\label{seqUlrich}
0\longrightarrow\cO_S(K_S)\longrightarrow\cF\longrightarrow\cI_{Z\vert S}(h_S)\longrightarrow0.
\end{equation}
Notice that
$$
c_1(\cF)=h_S+K_S,\qquad c_2(\cF)=\frac{h_S^2+h_SK_S}2+\chi(\cO_S)+1.
$$
\end{construction}

By definition $h^0\big(S,\cO_S(h_S)\big)=N+1$, $\rh^0\big(Z,\cO_Z\big)=N+2$ and $h^1\big(S,\cO_S(h_S)\big)=0$. Moreover, the choice of $Z$ implies $h^0\big(S,\cI_{Z\vert S}(h_S)\big)=0$. Thus the cohomology of the exact sequence
\begin{equation}
\label{seqStandard}
0\longrightarrow\cI_{Z\vert S}\longrightarrow\cO_S\longrightarrow\cO_Z\longrightarrow0.
\end{equation}
tensored by $\cO_S(h_S)$ and Equality \eqref{Serre} yield
\begin{equation}
\label{Extension}
\dim\Ext^1_S\big(\cI_{Z\vert S}(h_S),\cO_S(K_S)\big)=h^1\big(S,\cI_{Z\vert S}(h_S)\big)=1,
\end{equation}
i.e. $\cF$ is uniquely determined by $Z$.

\begin{lemma}
\label{lCohomology}
Let $S$ be a surface with $\kappa(S)\ge0$ and $q(S)=0$, endowed with a non--special very ample line bundle $\cO_S(h_S)$. Assume  $h^0\big(S,\cO_S(K_S-h_S)\big)=0$.

Then
$$
h^0\big(S,\cF\big)=p_g(S),\qquad h^0\big(S,\cF(-h_S)\big)=0,\qquad h^1\big(S,\cF\big)=0
$$
 for the vector bundle $\cF$ obtained from a scheme $Z$ as in Construction \eqref{conE}.
\end{lemma}
\begin{proof}
Since $h^0\big(S,\cI_{Z\vert S}\big)=0$ and $h^0\big(S,\cO_S(K_S-h_S)\big)=0$, it follows that the cohomology of Sequence \eqref{seqUlrich}  tensored by $\cO_S(-h_S)$ implies $h^0\big(S,\cF(-h_S)\big)=0$.

The vanishing $h^0\big(S,\cI_{Z\vert S}(h_S)\big)=0$ implies $h^0\big(S,\cF\big)=p_g(S)$. The equality $h^2\big(S,\cF\big)=h^0\big(S,\cF(-h_S)\big)=0$, the cohomology of Sequence \eqref{seqUlrich} and the vanishing $h^0\big(S,\cI_{Z\vert S}(h_S)\big)=0$ yield the exact sequence
$$
0\longrightarrow H^1\big(S,\cF\big)\longrightarrow H^1\big(S,\cI_{Z\vert S}(h_S)\big)\longrightarrow H^2\big(S,\cO_S(K_S)\big)\longrightarrow0,
$$
hence $h^1\big(S,\cF\big)= h^1\big(S,\cI_{Z\vert S}(h_S)\big)- h^2\big(S,\cO_S(K_S)\big)$.
Since  $h^2\big(S,\cO_S(K_S)\big)=1$, it follows that $h^1\big(S,\cF\big)=0$, thanks to Equality \eqref{Extension}. 
\end{proof}

In the next proposition we deal with the properties of the point corresponding to $\cF$ in the moduli spaces $\Spl_S(2;c_1(\cF),c_2(\cF))$ and $M_S(2;c_1(\cF),c_2(\cF))$.

To this purpose, we denote by $\mathcal Z_0$ the open subset of $\mathcal Z$ of points corresponding to schemes $Z$ such that $h^0\big(S,\cI_{Z\vert S}(K_S)\big)=0$. Trivially, $\mathcal Z_0\ne\emptyset$ if $N+2\ge p_g(S)$. It is immediate to check that such an inequality is equivalent to $h_S^2+4\ge h_SK_S$ if $h^0\big(S,\cO_S(K_S-h_S)\big)=0$.

Finally, let $\mathcal Z_1:=S_0^{[N+2]}\cap\mathcal Z$. As pointed out in Remark \ref{rS_0}, $\mathcal Z_1$ could be empty, but if $k$ is uncountable and $\kappa(S)\ge0$, then it is certainly dense inside $S^{[N+2]}$.

\begin{proposition}
\label{pStable}
Let $S$ be a surface with $q(S)=0$ and $p_g(S)\ge1$, endowed with a non--special very ample line bundle $\cO_S(h_S)$. Assume  $h^0\big(S,\cO_S(K_S-h_S)\big)=0$.

Then the following properties hold for the vector bundle $\cF$ obtained from a scheme $Z\in \mathcal Z$ as in Construction \eqref{conE}
\begin{enumerate}
\item $\cF$ is simple.
\item $p_g(S)\le \dim\Ext_S^2\big(\cF,\cF\big)\le p_g(S)+h^0\big(S,\cO_S(2K_S-h_S)\big)$ if $h_S^2+4\ge h_SK_S$ and $Z\in \mathcal Z_0$.
\item $\cF$ is $\mu$--stable if $h_S^2> h_SK_S$ and $Z\in \mathcal Z_1$.
\end{enumerate} 
\end{proposition}
\begin{proof}
In order to prove assertion (1), applying $\Hom_S\big(\cF,-\big)$ to Sequences \eqref{seqStandard} tensored by $\cO_S(h_S)$ and \eqref{seqUlrich}, taking into account of Lemma \ref{lCohomology} we obtain
\begin{align*}
\Hom_S(\cF,\cF\big)\subseteq\Hom_S(\cF,\cI_{Z\vert S}(h_S)\big)\subseteq\Hom_S(\cF,\cO_S(h_S)\big)\cong H^0\big(S,\cF(-K_S)\big).
\end{align*}
Tensoring the  cohomology of Sequence \eqref{seqUlrich} by $\cO_S(-K_S)$ we obtain
$$
h^0\big(S,\cF(-K_S)\big)\le1+h^0\big(S,\cI_{Z\vert S}(h_S-K_S)\big).
$$
The choice of $Z$ and the hypothesis $p_g(S)\ge1$ imply
$$
h^0\big(S,\cI_{Z\vert S}(h_S-K_S)\big)\le h^0\big(S,\cI_{Z\vert S}(h_S)\big)=0,
$$
hence $\cF$ is simple.

Let us prove assertion (2). The choice of $Z$ implies $h^0\big(S,\cI_{Z\vert S}(K_S)\big)=0$, then the cohomology of Sequence \eqref{seqUlrich} tensored by $\cO_S(K_S-h_S)$ yields 
$$
h^0\big(S,\cF(K_S-h_S)\big)=h^0\big(S,\cO_S(2K_S-h_S)\big).
$$
The cohomology of the same exact sequence tensored by $\cF(-h_S)\cong\cF^\vee(K_S)$ returns 
\begin{align*}
\dim \Ext_S^2\big(\cF,\cF\big)&=\rh^0\big(S,\cF\otimes\cF^\vee(K_S)\big)\le \\
&\le \rh^0\big(S,\cF\otimes\cI_{Z\vert S}\big)+h^0\big(S,\cO_S(2K_S-h_S)\big).
\end{align*}
The obvious inclusion $\cF\otimes\cI_{Z\vert S}\subseteq\cF$ yields $h^0\big(S,\cF\otimes\cI_{Z\vert S}\big)\le p_g(S)$ thanks to  Lemma \ref{lCohomology}, hence
\begin{align*}
\dim \Ext_S^2\big(\cF,\cF\big)\le p_g(S)+h^0\big(S,\cO_S(2K_S-h_S)\big).
\end{align*}
Assertion (2) follows by combining Inequality \eqref{Mukai} with the above inequality. 

Let us prove assertion (3). Thanks to \cite[Theorem II.1.2.2]{O--S--S}, if $\cF$ is not $\mu$--stable, then there should exist a sheaf $\mathcal M\subseteq\cF$ of rank $1$ such that $\cF/\cM$ is torsion--free and 
$$
\mu(\mathcal M)\ge\mu(\cF)=\frac{h_S^2+h_SK_S}2.
$$

The sheaf $\cM$ is trivially torsion--free and it is also normal (see \cite[Lemma II.1.1.16]{O--S--S}). Thus it is a line bundle, because it has rank $1$ (see \cite[Lemmas II.1.1.12 and II.1.1.15]{O--S--S}). It follows that $\cM\cong\cO_S(E)$ for some divisor $E$ on $S$.

We have
$$
(K_S-E)h_S\le \frac{h_SK_S-h_S^2}2<0.
$$
The Nakai criterion then  implies
\begin{equation}
\label{VanishingA}
h^0\big(S,\cO_S(K_S-E)\big)=0,
\end{equation}
hence $\cO_S(E)$ is not contained in the kernel $\cK\cong\cO_S(K_S)$ of the map $\cF\to\cI_{Z\vert S}(h_S)$ in Sequence \eqref{seqUlrich}
and the composition $\cO_S(E)\subseteq\cF\to\cI_{Z\vert S}(h_S)$ is necessarily non--zero. In particular $\rh^0\big(S,\cI_{Z\vert S}(h_S-E)\big)\ge1$, hence there is $A\in \vert h_S-E\vert$ through $Z$. 

We claim that $Ah_S\ge N+1$. Assuming the claim, Equality \eqref{Dimension} yields
$$
\frac{h_S^2-h_SK_S}2+p_g(S)+1\le Ah_S=(h_S-E)h_S\le\frac{h_S^2-h_SK_S}2,
$$
a contradiction. We deduce that a sheaf $\cM$ as above does not exist, hence $\cF$ is $\mu$--stable.

It remains to prove the claim. To this purpose, let $C_1,\dots, C_s$ be the integral components of $A$ intersecting $Z$ and $B$ their union. Since $Z\subseteq S_0$, it follows that $p_a(C_i)\ge1$.
Moreover, it is well known that $p_a(C_i\cup C_j)=p_a(C_i)+p_a(C_j)+C_iC_j-1$ for each $i, j\in\{\ 1,\dots, s\ \}$, $i\ne j$ (e.g. see \cite[Exercise V.1.3 (c)]{Ha2}): by combining these remarks with an easy induction on $s$ we then deduce 
\begin{equation}
\label{Positive}
p_a(B)\ge\sum_{i=1}^sp_a(C_i)-s+1\ge1.
\end{equation}

On the one hand, the cohomology of 
$$
0\longrightarrow\cO_S(h_S-B)\longrightarrow\cO_S(h_S)\longrightarrow\cO_B(h_B)\longrightarrow0
$$
yields $h^1\big(B,\cO_B(h_B)\big)\le h^2\big(S,\cO_S(h_S-B)\big)$. Vanishing \eqref{VanishingA} then implies
\begin{align*}
h^2\big(S,\cO_S(h_S-B)\big)&=h^0\big(S,\cO_S(B-h_S+K_S)\big)\le\\
&\le h^0\big(S,\cO_S(A-h_S+K_S)\big)=h^0\big(S,\cO_S(K_S-E)\big)=0,
\end{align*}
hence $h^1\big(B,\cO_B(h_B)\big)=0$. Thus the Riemann--Roch theorem for the curve $B$ and Inequality \eqref{Positive} above yield $h^0\big(B,\cO_B(h_B)\big)\le Bh_S$. 
On the other hand, the curve $B$ is not contained in any hyperplane inside $\p N$, because it contains $Z$, hence $h^0\big(B,\cO_B(h_B)\big)\ge N+1$. 

We then deduce that $Ah_S\ge Bh_S\ge N+1$, hence the claim is proved and the proof of assertion (3) is complete.
\end{proof}

\begin{remark}
Construction \ref{conE} makes sense also in the case $p_g(S)=0$. Indeed it is the method we used in \cite{Cs4,Cs4erratum} for proving the existence of special Ulrich bundles when $p_g(S)=q(S)=0$. Notice that in this case the inequality $h^2_S> h_SK_S$ is for free. 

On the one hand, the proof of Lemma \ref{lCohomology} can be carried over word by word to this case. On the other hand, the proof of assertion (1) of Proposition \ref{pStable} cannot be extended to the case $p_g(S)=0$. Moreover, the proof of assertion (3) is alternative to \cite[Theorem 1.2]{Cs4,Cs4erratum} when $\mathcal Z_1\ne\emptyset$: in particular, it certainly needs $\kappa(S)\ge0$. 

One of the hypothesis of \cite[Theorem 1.2]{Cs4,Cs4erratum} is that $k$ is uncountable. Thus, the above proof and \cite{L--M} extend such a result when $\kappa(S)=2$ also to the case of a countable base field. When $p_g(S)=0$ and $\kappa(S)\le1$, as in the case $p_g(S)\ge1$, the condition $\mathcal Z_1\ne\emptyset$ is not immediate: e.g. there exist Enriques surfaces containing infinitely many rational curves (see \cite{C--D}). 
\end{remark}

\section{The inductive step}
\label{sStep}
In this section we explain the inductive step of the proof of Theorem \ref{tMain}.

\begin{construction}
\label{conEeta}
Let $S$ be a surface endowed with a very ample line bundle $\cO_S(h_S)$. Let $\cF$ be a vector  bundle of rank $2$ such that $h^0\big(S,\cF\big)\ge1$. Thus for each point $p\in S$, there exist non--zero morphisms $\varphi\colon \cF\to\cO_p$.

Each $\varphi$ as above is surjective and we have the exact sequence
\begin{equation}
\label{seqDeform}
0\longrightarrow\cF_\varphi\longrightarrow\cF\mapright{\varphi}\cO_p\longrightarrow0,
\end{equation}
where $\cF_\varphi:=\ker(\varphi)$. Notice that
$$
c_1(\cF_\varphi)=c_1(\cF),\qquad c_2(\cF_\varphi)=c_2(\cF)+1,
$$
because $c_1(\cO_p)=0$ and $c_2(\cO_p)=-p$: see \cite[Example 15.3.1]{Fu}.
\end{construction}

\begin{lemma}
\label{lCohomologyEta}
Let $S$ be a surface endowed with a very ample line bundle $\cO_S(h_S)$. Assume that $\cF$ is a vector bundle of rank $2$ such that  $h^0\big(S,\cF(-h_S)\big)=0$ and $h^0\big(S,\cF\big)\ge1$.

If $p\in S$ and $\varphi\in \Hom_S\big(\cF,\cO_p\big)$ are general, then 
$$
h^0\big(S,\cF_\varphi\big)=h^0\big(S,\cF\big)-1,\qquad h^0\big(S,\cF_\varphi(-h_S)\big)=0,\qquad h^1\big(S,\cF_\varphi\big)=h^1\big(S,\cF\big)
$$
for the sheaf $\cF_\varphi$ obtained from $\cF$ as in Construction \eqref{conEeta}.
\end{lemma}
\begin{proof}
Trivially $h^0\big(S,\cF_\varphi(-h_S)\big)\le h^0\big(S,\cF(-h_S)\big)=0$. The equality $h^0\big(S,\cF_\varphi\big)=h^0\big(S,\cF\big)-1$ is equivalent to the surjectivity of the map $\overline{\varphi}\colon \rH^0\big(S,\cF\big)\to k$ induced by $\varphi$. The fact that $\overline{\varphi}$ is surjective for general $p\in S$ and $\varphi\in \Hom_S\big(\cF,\cO_p\big)$ is an open property, thus it suffices to check the existence of at least one $p$ and one $\varphi$ with such a property.

To this purpose there exists a non zero $s\in H^0\big(S,\cF\big)$. Its zero locus is the union of a 0--dimensional scheme $X$ and a divisor $E$ on $S$. In particular we have an exact sequence of the form
$$
0\longrightarrow \cO_S(E)\mapright\varepsilon \cF\longrightarrow \cI_{X\vert S}(c_1(\cF)-E)\longrightarrow0,
$$

Let now $p\not\in X\cup E$. Applying $\Hom_S\big(\cI_{X\vert S}(c_1(\cF)-E),-\big)$ to the exact sequence
\begin{equation}
\label{seqGood}
0\longrightarrow \cI_{p\vert S}(E)\longrightarrow \cO_S(E)\longrightarrow \cO_p\longrightarrow0,
\end{equation}
we obtain the exact sequence
\begin{align*}
\Ext_S^1\big(\cI_{X\vert S}(c_1(\cF)-E),\cI_{p\vert S}(E)\big)&\mapright\delta\Ext_S^1\big(\cI_{X\vert S}(c_1(\cF)-E),\cO_S(E)\big)\longrightarrow\\
&\longrightarrow\Ext_S^1\big(\cI_{X\vert S}(c_1(\cF)-E),\cO_p\big).
\end{align*}

We have
\begin{gather*}
\sHom_S\big(\cI_{X\vert S}(c_1(\cF)-E),\cO_p\big)\cong\cO_p,\\
 \sExt_S^1\big(\cI_{X\vert S}(c_1(\cF)-E),\cO_p\big)=0
\end{gather*}
because we compute them locally on $X$ and $p\not\in X\cup E$. The exact sequence of the low degree terms of the spectral sequence
$$
E^{i,j}_2:=H^i\big(S,\sExt_S^j\big(\cI_{X\vert S}(c_1(\cF)-E),\cO_p\big)\big)
$$
then yields $\Ext_S^1\big(\cI_{X\vert S}(c_1(\cF)-E),\cO_p\big)=0$. It follows that the map $\delta$ above is surjective, hence we can lift Sequence \eqref{seqDeform} to another exact sequence obtaining a commutative diagram of the form
\begin{equation*}
\begin{CD}
@. 0@.   0\\
@.@VVV  @VVV@.\\
0@>>> \cI_{p\vert S}(E)@>>>   \cF_\varphi@>>> \cI_{X\vert S}(c_1(\cF)-E)@>>>0\\
@.@VVV  @VVV@|\\
0@>>> \cO_S(E)@>>>   \cF@>>>\cI_{X\vert S}(c_1(\cF)-E)@>>>0\\
@.@VV\epsilon V  @VV\varphi V@.\\
@. \cO_p@=   \cO_p\\
@.@VVV  @VVV@.\\
@. 0@.   0
\end{CD}
\end{equation*}

Since $p\not\in E$, it follows that $h^0\big(S,\cI_{p\vert S}(E)\big)=h^0\big(S,\cO_S(E)\big)-1$, hence the map $\overline{\epsilon}\colon H^0\big(S,\cO_S(E)\big)\to k$ induced by $\epsilon$ is surjective. The commutativity of the diagram then implies that the same is true for $\overline{\varphi}$. Thus $h^0\big(S,\cF_\varphi\big)=h^0\big(S,\cF\big)-1$ and $h^1\big(S,\cF_\varphi\big)=h^1\big(S,\cF\big)$ for  $\varphi$ general enough. 
\end{proof}

\begin{proposition}
\label{pStableEta}
Let $S$ be a surface endowed with a very ample line bundle $\cO_S(h_S)$. Assume that $\cF$ is a vector bundle of rank $2$ with $h^0\big(S,\cF\big)\ge1$.

Then the following properties hold for a sheaf $\cF_\varphi$ obtained from $\cF$ as in \eqref{conEeta}
\begin{enumerate}
\item $\cF_\varphi$ is simple if the same is true for $\cF$.
\item $p_g(S)\le \dim\Ext_S^2\big(\cF_\varphi,\cF_\varphi\big)\le \dim\Ext_S^2\big(\cF,\cF\big)$.
\item $\cF_\varphi$ is $\mu$--stable if the same is true for $\cF$.
\end{enumerate}
\end{proposition}
\begin{proof}
In order to prove assertion (1), we notice that the quotient map $\cO_S\to\cO_p$ induces by duality an inclusion $0\ne \Hom_S\big(\cO_p,\cO_p\big)\subseteq\Hom_S\big(\cO_S,\cO_p\big)\cong k$, hence
\begin{equation}
\label{U1}
\Hom_S\big(\cO_p,\cO_p\big)\cong k.
\end{equation}
Equality \eqref{Serre} yields
$$
\dim\Ext_S^i\big(\cO_p,\cF\big)=h^{2-i}\big(S,\cF^\vee(K_S)\otimes\cO_p\big)={2\delta_{2,i}}.
$$
Thus the functor $\Hom_S\big(\cO_p,-\big)$ applied to Sequence \eqref{seqDeform} we obtain
\begin{equation}
\label{U2}
\Ext_S^1\big(\cO_p,\cF_\varphi\big)\cong \Hom_S\big(\cO_p,\cO_p\big). 
\end{equation}

On the one hand, each map in $\Hom_S\big(\cF,\cF_\varphi\big)$ induces a map in $\Hom_S\big(\cF,\cF\big)$ by composing with the inclusion $\cF_\varphi\subseteq\cF$: it follows that the composed map is never surjective. On the other hand, each non--zero element in $\Hom_S\big(\cF,\cF\big)$ must be a homothety because $\cF$ is simple, hence it is an isomorphism. We deduce that $\Hom_S\big(\cF,\cF_\varphi\big)=0$. 

Thus we obtain $\Hom_S\big(\cF_\varphi,\cF_\varphi\big)\subseteq \Ext_S^1\big(\cO_p,\cF_\varphi\big)$, by applying $\Hom_S(-,\cF_\varphi\big)$ to Sequence \eqref{seqDeform}. If we combine this inclusion with Equality \eqref{U1} and \eqref{U2}, we finally obtain that $\cF_\varphi$ is simple, i.e. assertion (1) is proven.

We prove assertion (2). Since $\cF$ is a vector bundle, the following obvious equalities
$$
h^0\big(S,\cF^\vee\otimes\cO_p\big)=2, \qquad h^1\big(S,\cF^\vee\otimes\cO_p\big)=h^2\big(S,\cF^\vee\otimes\cO_p\big)=0
$$
hold, hence $\Ext_S^2\big(\cF,\cF_\varphi\big)\cong\Ext_S^2\big(\cF,\cF\big)$, by applying $\Hom_S\big(\cF,-\big)$ to Sequence \eqref{seqDeform}.
Equality \eqref{Serre} and $\Hom_S\big(\cF_\varphi,-\big)$ applied to Sequence \eqref{seqDeform} tensored by $\cO_S(K_S)$ yield
\begin{equation*}
\label{inclusion}
\begin{aligned}
\Ext_S^2\big(\cF_\varphi,\cF_\varphi\big)^\vee\cong\Hom_S\big(\cF_\varphi,\cF_\varphi(K_S)\big)&\subseteq \Hom_S\big(\cF_\varphi,\cF(K_S)\big)\cong\\ &\cong\Ext_S^2\big(\cF,\cF_\varphi\big)^\vee\cong\Ext_S^2\big(\cF,\cF\big)^\vee.
\end{aligned}
\end{equation*}
The statement follows from the above inclusion and Inequality \eqref{Mukai}.

Consider now assertion (3). 
Since $\cF_\varphi\subseteq\cF$ and $\mu(\cF_\varphi)=\mu(\cF)$, it follows that each subsheaf destabilizing $\cF_\varphi$ also destabilizes $\cF$. Thus assertion (3) is proven.
\end{proof}

\section{The proof of Theorem \ref{tMain}}
\label{sMain}
In this section we put together the above results with the following classical theorem for proving Theorem \ref{tMain} stated in the introduction.

\begin{theorem}
\label{tAr}
Let $S$ be a surface, $p\in S$, $\cF$ a simple vector bundle of rank $r\ge2$ with $\Ext_S^2\big(\cF,\cF\big)=p_g(S)$, $\varphi\in\Hom_S\big(\cF,\cO_p\big)$ non--zero and $\cF_\varphi:=\ker(\varphi)$.

Then $\cF_\varphi$ has a universal deformation whose general sheaf is locally free at $p$.
\end{theorem}
\begin{proof}
See \cite[Theorem 1.4 and Corollary 1.5]{Ar}. The simplicity hypothesis is mentioned in \cite[p. 450]{Ar}.
\end{proof}

As pointed out in the Introduction, if  $\cE$ is a vector bundle on $S$ such that $c_1(\cE)=3h_S+K_S$, then Equality \eqref{Serre} implies that $\cE$ is Ulrich if and only if
$$
h^0\big(S,\cE(-h_S)\big)=h^1\big(S,\cE(-h_S)\big)=0.
$$
Let $\cF$ be the bundle defined in Construction \ref{conE}. If $p_g(S)\ge1$, then $\cE:=\cF(h_S)$ satisfies all the conditions for being a special Ulrich bundle, but the vanishing $h^0\big(S,\cE(-h_S)\big)=h^0\big(S,\cF\big)=0$. Nevertheless, $\cF$ can be viewed as a good approximation of a special Ulrich bundle. For this reason we introduce the following definition. 

\begin{definition}
Let $S$ be a surface endowed with a very ample line bundle $\cO_S(h_S)$. A torsion--free coherent sheaf $\cF$  of rank $2$ on $S$ is called {\sl $d$--good} if $c_1(\cF)=c_1:=h_S+K_S$, $h^0\big(S,\cF\big)=d$ and  $h^1\big(S,\cF\big)=h^0\big(S,\cF(-h_S)\big)=0$.
\end{definition}

As explained above, if $\cF$ is a  $0$--good bundle, then $\cE:=\cF(h_S)$ is a special Ulrich bundle of rank $2$. 

If $q(S)=0$, then the bundle $\cF$ obtained from a scheme $Z$ as in Construction \ref{conE} is $p_g(S)$--good by Lemma \ref{lCohomology}: in particular if $p_g(S)=0$, then $\cF(h_S)$ is a special Ulrich bundle.

\begin{remark}
\label{rGood}
Notice that, thanks to Equality \eqref{RRGeneral}, a vector bundle $\cF$ of rank $2$ is $d$--good if and only if 
$$
 c_1(\cF)=c_1,\qquad c_2(\cF)=c_2^{(d)}:=\frac{h_S^2+h_SK_S}2+2\chi(\cO_S)-d
$$
 and  $h^1\big(S,\cF\big)=h^0\big(S,\cF(-h_S)\big)=0$, because $h^2\big(S,\cF\big)=h^0\big(S,\cF(-h_S)\big)$.
\end{remark}

We are now able to prove Theorems \ref{tMain} stated in the introduction.

\medbreak
\noindent{\it Proof of Theorem \ref{tMain}.}
We will prove the statement by descending induction, showing for each $d=p_g(S),\dots,0$ the existence of a simple (resp. $\mu$--stable) $d$--good bundle $\cF_d$ with $\dim\Ext_S^2\big(\cF_d,\cF_d\big)=p_g(S)$. 

For the base step let $\cF_{p_g(S)}$ be the bundle $\cF$ defined in Construction \eqref{conE} from $Z\in\mathcal Z$ (resp. $Z\in\mathcal Z_1$) which is ${p_g(S)}$--good and simple when $h_S^2+4\ge h_SK_S$ (resp. $\mu$--stable when $h_S^2>h_SK_S$), thanks to Lemma \ref{lCohomology} and Proposition \ref{pStable}. 

Now let $p_g(S)\ge d\ge 1$ and assume the existence of a simple (resp. $\mu$--stable) $d$--good bundle $\cF_d$ with $\dim\Ext_S^2\big(\cF_d,\cF_d\big)=p_g(S)$. It follows from Construction \ref{conEeta} the existence of a simple (resp. $\mu$--stable) $(d-1)$--good sheaf $\cF_{\varphi}$  with $\dim\Ext_S^2\big(\cF_{\varphi},\cF_{\varphi}\big)=p_g(S)$, thanks to Lemmas \ref{lCohomology}, \ref{lCohomologyEta} and Proposition \ref{pStableEta}. 

Thanks to \cite[Theorem 0.3]{Muk}, the point corresponding to $\cF_{\varphi}$ in $\Spl_S(2;c_1,c_2^{d-1})$ (resp. $M_S(2;c_1,c_2^{d-1})$) is smooth and it lies in a component $M$ of dimension
$$
\dim\Ext_S^1\big(\cF_{\varphi},\cF_{\varphi}\big)=h_S^2-K_S^2+5\chi(\cO_S)-4d+4. 
$$
Thus we have an integral open smooth neighborhood $B\subseteq M$ of the point corresponding to $\cF_{\varphi}$ and a flat family $\mathfrak F$ over $B$ with  ${\mathfrak F}_{b_0}\cong\cF_{\varphi}$ for some $b_0\in B$. 

Since $\cF_d$ satisfies the hypothesis of Theorem \ref{tAr} and $\cF_{\varphi}$ is locally free on $S\setminus\{\ p\ \}$, up to shrinking $B$ we can assume that  $\mathfrak F_b$ is locally free at $p$ for each $b\in B\setminus\{\ b_0\ \}$. Thus the sheaf $\cF_{d-1}:=\mathfrak F_b$ is a simple (resp. $\mu$--stable) $(d-1)$--good bundle  for each choice $b\in B\setminus\{\ b_0\ \}$. 

The smoothness of $B$ implies that the tangent space at $b\in B$ has dimension
\begin{align*}
h_S^2-K_S^2+4\chi(\cO_S)&-4d+5+\dim\Ext^2_S\big(\cF_{d-1},\cF_{d-1}\big)=\\
&=\dim\Ext^1_S\big(\cF_{d-1},\cF_{d-1}\big)=\dim\Ext^1_S\big(\cF_{\varphi},\cF_{\varphi}\big),
\end{align*}
hence $\dim\Ext_S^2\big(\cF_{d-1},\cF_{d-1}\big)=p_g(S)$. 

We can repeat the above steps untill $d=1$. After $p_g(S)$ steps we obtain the existence of a $0$--good bundle $\cF_0$ representing a smooth point in $\Spl_S(2;c_1,c_2^{(p_g(S))})$ (resp. $M_S(2;c_1,c_2^{(p_g(S))}$) lying in a component of dimension $h_S^2-K_S^2+5\chi(\cO_S)$, which is necessarily unique. 

In order to prove the theorem we finally notice that the map $\cG\mapsto\cG(h_S)$ induces a well defined isomorphism $\Spl_S(2;c_1,c_2^{(p_g(S))})\cong \Spl_S(2;c_1^{sp},c_2^{sp})$ (resp. $M_S(2;c_1,c_2^{(p_g(S))})\cong  M_S(2;c_1^{sp},c_2^{sp})$).
\qed
\medbreak

\section{Examples}
\label{sExamples}
In this section we give some applications of Theorem \ref{tMain}. We first show that Ulrich bundles are quite common on regular surfaces. Then we give examples of polarized surfaces fulfilling the hypothesis of Theorem \ref{tMain}. 

\begin{example}
\label{eLarge}
Let $S$ be a surface with $q(S)=0$, $p_g(S)\ge1$ and $\cO_S(H)$ an ample line bundle. Then there exists an integer $n_0$ such that for each $n\ge n_0$ the surface $S$ supports a special Ulrich bundle of rank $2$ with respect to $\cO_S(nH)$. If $\mathcal Z_1\ne\emptyset$ such bundles can be taken $\mu$--stable.

Indeed, it suffices to check that there is $n_0$ such that $\cO_S(nH)$ satisfies the hypothesis of Theorem \ref{tMain} for $n\ge n_0$, i.e. $n^2H^2>nHK_S$ and $h^0\big(S,\cO_S(2K_S-nH)\big)=0$.

Since $H$ is ample, it follows that there is $n_1$ such that $n_1H$ is very ample and non--special. Again the ampleness of $H$ implies $H^2\ge1$, hence there is $n_2$ such that $(2K_S-n_2H)H<0$, hence $h^0\big(S,\cO_S(2K_S-n_2H)\big)=0$ by the Nakai criterion. Finally, it is easy to check that for such an $n_2$ we also have   $(n_2H-K_S)H>0$. It then suffices to chose $n_0:=\max\{\ n_1,n_2\ \}$.

The case when there exists a relatively minimal elliptic fibration without multiple fibres $\psi \colon S\to \p1$ is of particular interest. Indeed, in this case we have $\kappa(S)\le1$, thus $h^0\big(S,\cO_S(2K_S-h_S)\big)=0$ as pointed out in the introduction. 

We know that $\cO_S(K_S)\cong\psi^*\cO_{\p1}(e-2)$ where $\cO_{\p1}(e):=(R^1\psi_*\cO_S)^\vee$ (see \cite[Theorem V.12.1]{B--H--P--V}): in particular $\kappa(S)\ge0$ implies $e\ge2$ and we will assume such a restriction in what follows, hence $HK_S\ge0$.  Moreover, $p_g(S)=h^0\big(\p1,\cO_{\p1}(e-2)\big)$ and $\chi(\cO_S)=e$, hence $q(S)=0$ (see \cite[Proposition V.12.2]{B--H--P--V} and its proof). 

Let $\cO_S(H)$ be a very ample line bundle and set $\cO_S(h_S):=\cO_S(H+K_S)$.  Thus $h^1\big(S,\cO_S(h_S)\big)=0$ by the Kodaira vanishing theorem and $\cO_S(h_S)$ is very ample because $\cO_S(K_S)$ is globally generated. Finally, $h_S^2= H^2+2HK_S>HK_S=h_SK_S$ because $HK_S\ge0$.

An example of the above set up is the case of an elliptic fibration with a section: in this case there are trivially no multiple fibres. This case has been inspected in \cite{MR--PL2} with a more direct approach.
\end{example}

In \cite{Fa} the author proved the existence of special Ulrich bundles of rank $2$ on each polarized $K3$ surface, extending some earlier results (see \cite{C--K--M, A--F--O}). The case of Enriques surfaces was examined in \cite{Bea2,Cs4}. We list below some other interesting examples.

\begin{example}
\label{eGeneral}
Let $S$ be a regular surface and $\cO_S(h_S)\cong\cO_S(n K_S)$ for some integer $n\ge3$. Then $S$ supports a $\mu$--stable special Ulrich bundle of rank $2$ with respect to $\cO_S(h_S)$.

Indeed, in this case $\mathcal Z_1\ne\emptyset$ (see \cite{L--M}) and  $K_S$ is ample, hence $S$ is minimal: then $h^1\big(S,\cO_S(h_S)\big)=0$ thanks to \cite[Proposition VII.5.3]{B--H--P--V}. Moreover, the conditions $h_S^2>h_SK_S$ and $h^0\big(S,\cO_S(2K_S-h_S)\big)=0$ are trivially satisfied in this case. 

The above result cannot be extended to the cases $1\le n\le2$ using the same argument. Indeed,  $h^0\big(S,\cO_S(2K_S-h_S)\big)\ge1$ in these cases. In particular, we cannot apply Theorem \ref{tAr} in the base case of the induction in the proof of Theorem \ref{tMain}. 

Notice that the case $n=2$ could be within reach: indeed it would be sufficient to check that $h^0\big(S,\cF\otimes\cI_{Z\vert S}\big)<h^0\big(S,\cF\big)$ in the proof of Proposition \ref{pStable}, because $h^0\big(S,\cO_S(2K_S-h_S)\big)=1$ (for some similar results in this direction see Example \ref{eRa}). 

On the other hand, the case $n=1$ seems to be out of reach with our methods, because $h^0\big(S,\cO_S(2K_S-h_S)\big)=p_g(S)$ in this case.
\end{example}

It is classically known (see \cite{Bu} and the references therein) that non--degenerate surfaces $S$ of degree $d=2N-2+s$ in $\p N$ are geometrically ruled or $K3$ when $s\le0$. In \cite{Bu} the author gives a description of surfaces with $1\le s\le N-3$ when $\cO_S(h_S)$ is non--special.
In all these cases $s\equiv h_SK_S \pmod2$,
$$
h_SK_S\le \left\lbrace\begin{array}{ll} 
3s\quad&\text{if $0\le s\le N-3$,}\\
s-2\quad&\text{if $s<0$.}
\end{array}\right.
$$

\begin{example}
\label{eBu}
Let $k\cong \bC$ and $S\subseteq\p N$ a non--ruled, non--degenerate surface of degree $2N-2+s$ where $N\ge3+s$ and $0\le s\le 3$ such that $\cO_S(h_S)$ is non--special. Moreover, if $p_g(S)\ge1$ and $\kappa(S)=2$, we also assume
\begin{equation}
\label{Buium}
N\ge\frac{5s+3}2.
\end{equation}
Then $S$ supports special Ulrich bundles of rank $2$ with respect to $\cO_S(h_S)$. Since $k\cong \bC$, it follows that $\mathcal Z_1\ne\emptyset$ hence such bundles can be taken $\mu$--stable.

The classification given in \cite[Propositions 2.1, 2.2, 2.3]{Bu} and \cite[Lemma 1.10]{Bu} imply that $S$ is always regular in these cases, hence the assertion follows from \cite{Cs4,Cs4erratum} if $p_g(S)=0$. Thus, from now on, we will assume $p_g(S)\ge1$. Notice that the inequality $N\ge3+s$ implies $h_S^2> h_SK_S$.

If $\kappa(S)\le1$, then $h^0\big(S,\cO_S(2K_S-h_S)\big)=0$ trivially. If $\kappa(S)=2$, then Inequality \eqref{Buium} implies $2h_SK_S-h_S^2<0$, hence $h^0\big(S,\cO_S(2K_S-h_S)\big)=0$ thanks to the Nakai criterion. In both the cases the assertion then follows from  Theorem \ref{tMain}.
\end{example}

\begin{example}
\label{eElliptic}
Let $S\subseteq\p N$ be a non--degenerate regular surface with $\kappa(S)\le1$, $N\ge p_g(S)-2$  and such that $\cO_S(h_S)$ is non--special. Then $S$ supports simple special Ulrich bundles of rank $2$ with respect to $\cO_S(h_S)$.  If $\mathcal Z_1\ne\emptyset$ and $N\ge p_g(S)+1$ such bundles can be taken $\mu$--stable.

Indeed, in this case, we already know that
$$
h^0\big(S,\cO_S(K_S-h_S)\big)=h^0\big(S,\cO_S(2K_S-h_S)\big)=0.
$$
The inequalities $N\ge p_g(S)-2$ and $N\ge p_g(S)+1$ are respectively equivalent to $h_S^2+4\ge h_SK_S$ and $h_S^2> h_SK_S$.  Thus the assertions follows from Theorem \ref{tMain}.

As a more concrete example, assume that $k\cong \bC$ and let $S\subseteq\p4$ be a surface with $\kappa(S)=1$ and degree $d=7,8$. In \cite{Ok2,Ok3} it is shown that  the surfaces $S$ are exactly the ones linked to a plane in either a quadro--quartic or a cubo--cubic complete intersection inside $\p4$.

The results in \cite{P--S} imply that the minimal free resolutions of $\cI_{S\vert\p4}$ look like
\begin{gather*}
0\longrightarrow \cO_{\p4}(-5)^{\oplus2}\longrightarrow \cO_{\p4}(-2)\oplus\cO_{\p4}(-4)^{\oplus2}\longrightarrow \cI_{S\vert\p4}\longrightarrow 0,\\
0\longrightarrow \cO_{\p4}(-5)^{\oplus2}\longrightarrow \cO_{\p4}(-3)^{\oplus2}\oplus\cO_{\p4}(-4)\longrightarrow \cI_{S\vert\p4}\longrightarrow 0
\end{gather*}
respectively. The cohomology of the exact sequence
$$
0\longrightarrow\cI_{S\vert \p 4}\longrightarrow\cO_{\p 4}\longrightarrow\cO_S\longrightarrow0
$$
implies that $\cO_S(h_S)$ is non--special, $q(S)=0$, $p_g(S)=2$.


In both the cases $S$ is determinantal. Nevertheless, $S$ is not defined by a matrix with linear entries, hence we cannot use the results proved in \cite{K--MR1, K--MR2} for deducing the existence of an Ulrich bundle.

Similarly, we cannot use the results in \cite{MR--PL2}.
Indeed, in \cite[Theorem III.4.2 and Observation III.3.5]{Lop}, the author proves that for a surface $S$ as above which is also very general, then  $\Pic(S)$ is generated by $h_S$ and $K_S$. 
If the canonical map $\psi \colon S\to \p1$ has a section $\sigma\colon\p1\to S$, then $C:=\im(\sigma)$ is a rational curve linearly equivalent to $xh_S+yK_S$ for some integers $x,y$. Since $C$ is the image of a section of the canonical map, it follows that $1=CK_S=xh_SK_S$ which is impossible, because the last number is a multiple of $3$ or $4$ (according with the two cases $h_S^2=7,8$).
\end{example}

\section{Ulrich--wildness}
\label{sWild}
In this very short section we deal with the size of the families of Ulrich bundles supported on the surfaces we are interested in.

A variety $X$ is called {\sl Ulrich--wild} if it supports families of dimension $p$ of pairwise non--isomorphic, indecomposable, Ulrich sheaves with respect to $\cO_X(h_X)$ for arbitrary large $p$. 

\begin{lemma}
\label{lFPL}
Let $X$ be a smooth variety endowed with a very ample line bundle $\cO_X(h_X)$.

If $\cE$ is a simple Ulrich bundle on $X$ such that $\dim \Ext_S^1\big(\cE,\cE\big)\ge3$,
then $X$ is Ulrich--wild.
\end{lemma}
\begin{proof}
It is an immediate consequence of  \cite[Theorem A, Corollary 2.1 and Remark 1.6 v)]{F--PL}, because every Ulrich bundle is semistable by \cite[Theorem 2.9]{C--H2} and each non--zero automorphism of a simple sheaf is an isomorphism.
\end{proof}

We deduce the following criterion for the surfaces we are interested in. 

\begin{proposition}
\label{pWild}
Let $S$ be a surface with $q(S)=0$, $p_g(S)\ge1$, endowed with a very ample non--special line bundle $\cO_S(h_S)$. Assume $h^0\big(S,\cO_S(2K_S-h_S)\big)=0$ and $h_S^2+4\ge h_SK_S$.

Then $S$ is Ulrich--wild.
\end{proposition}
\begin{proof}
Consider the bundle  $\cF$ defined in Construction \eqref{conE}. We know from Proposition \ref{pStable} that $\cF$ is simple and $\dim\Ext_S^1\big(\cF,\cF\big)=p_g(S)$. Then, Equality \eqref{RRGeneral} yields
$$
h_S^2-K_S^2+\chi(\cO_S)+4=1+p_g(S)-\chi(\cF\otimes\cF^\vee)=\dim\Ext_S^1\big(\cF,\cF\big)\ge0.
$$
The same argument applied to the Ulrich bundle $\cE$ defined in Theorem \ref{tMain}, then yields
\begin{equation*}
\label{Bound}
\dim\Ext_S^1\big(\cE,\cE\big)=1+p_g(S)-\chi(\cE\otimes\cE^\vee)=h_S^2-K_S^2+5\chi(\cO_S)\ge3.
\end{equation*}
Thus the statement follows immediately from Lemma \ref{lFPL}.
\end{proof}

The above proposition extends \cite[Theorem 1.3]{Cs4} to the case $p_g(S)\ge1$.  

\begin{corollary}
Let $S$ be one of the surfaces described in Examples \ref{eLarge}, \ref{eGeneral}, \ref{eBu}, \ref{eElliptic}.

Then $S$ is Ulrich--wild.
\end{corollary}
\begin{proof}
The statement is a trivial consequence of Proposition \ref{pWild}.
\end{proof}

\section{Ulrich bundles on surfaces of low degree}
\label{sLow}
In \cite[Theorems 1.4 and 1.5]{Cs5} the author deals with the existence of special Ulrich bundles on surfaces of low degree on surfaces $S\subseteq\p N$. 

We start this section by improving \cite[Theorem 1.4]{Cs5}. We work over the complex field $\bC$, hence $S_0$ is dense for each surface $S$ with $\kappa(S)\ge0$ by Remark \ref{rS_0}.

\begin{theorem}
\label{tLowDegree}
Let $k\cong\bC$, $S\subseteq\p N$ a surface of degree $d\le8$.

Then the following assertions holds.
\begin{enumerate}
\item If $\kappa(S)\le1$, then $S$ supports special Ulrich bundles.
\item If $\kappa(S)=2$ and $S$ is general, then $S$ supports special Ulrich bundles.
\end{enumerate}
\end{theorem}
\begin{proof}
Thanks to \cite[Theorem 1.4]{Cs5} we know that if $\kappa(S)\ne1$, then $S$ supports Ulrich bundles $\cE$ of even rank $r^{sp}_{Ulrich}$ as in \cite[Table A]{Cs5} and such that
$$
c_1(\cE)= \frac{r^{sp}_{Ulrich}}2(3h_S +K_S).
$$
In particular $r^{sp}_{Ulrich}=2$ if either $\kappa(S)\le0$ or $\kappa(S)=2$ and $S$ is general. Thus such an Ulrich bundle $\cE$ is special in these cases.

If $\kappa(S)=1$, taking into account of the classification in  \cite[Table A]{Cs5} and the results in \cite{Ok2,Ok3}, we know that $S\subseteq\p 4$ is a properly elliptic surface of degree either $7$ or $8$. The existence of special Ulrich bundles on such surfaces has been proved in Example \ref{eElliptic}. 
\end{proof}

We now extend \cite[Theorem 1.5]{Cs5}. Here $\pi(S)$ denotes the genus of a general plane section of $S$.

\begin{theorem}
\label{tLowWild}
Let $k\cong\bC$, $S\subseteq\p N$ a surface of degree $d\le 8$.

Then $S$ is Ulrich--wild if and only if either $d\ge5$, or $d\le4$ and  $\pi(S)\ge1$.
\end{theorem}
\begin{proof}
The statement coincides with \cite[Theorem 1.5]{Cs5} when $\kappa(S)\ne1$. 

If $\kappa(S)=1$, then we know (see the proof of Theorem \ref{tLowDegree}) that $S$ is regular, non--special, $h_S^2+4\ge h_SK_S$ and $h^0\big(S,\cO_S(2K_S-h_S)\big)=0$. Then the statement follows from Proposition \ref{pWild}.
\end{proof}

We close the section with some  examples dealing with non--degenerate surfaces of degree $d>8$ in $\p4$.

\begin{example}
\label{eAR}
Let $k\cong\bC$ and $S\subseteq\p4$ a non--degenerate surface of degree $9$ with non--special $\cO_S(h_S)$ and $h_S^2>h_SK_S$. 

The surfaces as above are regular, thanks to \cite[Theorem 0.1]{A--R}. If $p_g(S)=0$, then $S$ supports special Ulrich bundles which are $\mu$--stable thanks to \cite{Cs4,Cs4erratum}, because $S$ is neither a rational normal scroll nor a plane by degree reasons. In what follows we will briefly deal with the case $p_g(S)\ge1$.

We have $h^2\big(S,\cO_S(h_S)\big)=0$ because $9=h_S^2> h_SK_S$. The so--called Severi Theorem (see \cite{Se}) implies $h^0\big(S,\cO_S(h_S)\big)=5$. Since $\cO_S(h_S)$ is non--special, it follows from Equality \ref{RRGeneral} that $h_SK_S=2\chi(\cO_S)-1$. Moreover, the adjunction formula returns $h_SK_S=2\pi(S)-11$. By combining such two equalities we obtain $\chi(\cO_S)=\pi(S)-5$. Finally, the inequality $9> h_SK_S$ forces $\pi(S)\le9$.

Thanks to the above discussion and to the classification in \cite[Theorem 0.1 and its proof]{A--R} we have to deal only with the following cases.
\begin{itemize}
\item A minimal properly elliptic surface $S$ with $h_SK_S=3$, $K_S^2=0$, $p_g(S)=1$.
\item A minimal surface $S$ of general type with $h_SK_S=5$, $K_S^2=1$, $p_g(S)=2$.
\item A surface $S$ linked with a possibly singular/reducible cubic scroll $Y$ via a cubo--quartic complete intersection: in this case $h_SK_S=7$, $K_S^2=2$, $p_g(S)=3$.
\end{itemize}

In the first case  case we trivially have $h^0\big(S,\cO_S(2K_S-h_S)\big)=0$, hence Theorem \ref{tMain} yields the existence of $\mu$--stable special Ulrich bundles on $S$.

In the remaining cases the vanishing $h^0\big(S,\cO_S(2K_S-h_S)\big)=0$ is not evident as in the previous case. E.g., let us examine the second case. If $h^0\big(S,\cO_S(2K_S-h_S)\big)\ne0$, then $h^0\big(S,\cO_S(2K_S)\big)\ge h^0\big(S,\cO_S(h_S)\big)=5$. Moreover, $h^2\big(S,\cO_S(2K_S)\big)=0$ by the Serre duality. Thus Equality \eqref{RRGeneral} implies
$$
h^0\big(S,\cO_S(2K_S)\big)-h^1\big(S,\cO_S(2K_S)\big)=K_S^2+\chi(\cO_S).
$$
It follows that $h^1\big(S,\cO_S(2K_S)\big)\ge1$, hence $S$ would not be minimal thanks to \cite[Corollary VII.5.4]{B--H--P--V}, a contradiction. Thus $h^0\big(S,\cO_S(2K_S-h_S)\big)=0$, hence the existence of $\mu$--stable special Ulrich bundles on $S$ still follows from Theorem \ref{tMain}.

In the third case the same argument does not hold. Indeed $\kappa(S)\ge1$ and $K_S^2=2$. The minimal model $S_{min}$ of $S$ satisfies $K_{S_{min}}^2\ge2$, thus $\kappa(S)=2$. Nevertheless, the vanishing $h^0\big(S,\cO_S(2K_S-h_S)\big)=0$ can be proved also in this case. Indeed, in \cite[Section (2.13)]{A--R} the authors show that the general hyperplane section $C$ of the surface $Y$ linked to $S$ is aCM with $p_a(C)=0$. This fact and the Riemann--Roch theorem on $C$ implies that the ideal of $C$, hence of $Y$, is generated by the minors of a $3\times2$ matrix of linear forms. Thus there exists an exact sequence of the form
$$
0\longrightarrow\cO_{\p4}(-3)^{\oplus2}\longrightarrow\cO_{\p4}(-2)^{\oplus3}\longrightarrow\cI_{Y\vert\p4}\longrightarrow0.
$$
Thanks to \cite{P--S} there is a resolution of the form
$$
0\longrightarrow\cO_{\p4}(-5)^{\oplus3}\longrightarrow\cO_{\p4}(-3)\oplus\cO_{\p4}(-4)^{\oplus3}\longrightarrow\cI_{X\vert\p4}\longrightarrow0.
$$
Thus \cite[Proposition II.2.4]{DA} implies $h^0\big(S,\cO_S(2K_S-h_S)\big)=0$, hence Theorem \ref{tMain} yields the existence of $\mu$--stable special Ulrich bundles on $S$. 
\end{example}

\begin{example}
\label{eRa}
Let $k\cong\bC$ and $S\subseteq\p4$ a non--degenerate surface of degree $10$ with non--special $\cO_S(h_S)$ and $h_S^2> h_SK_S$. 

Surfaces as above are classified in \cite[Theorem 0.1]{Ra}. If $q(S)=p_g(S)=0$, then $S$ is not a rational normal scroll by degree reasons, hence it supports special Ulrich bundles which are $\mu$--stable from  \cite{Cs4,Cs4erratum}.

From now on we will assume that $q(S)$ and $p_g(S)$ do not vanish simultaneously. The same argument of Example \ref{eAR} yields $h_SK_S=2\chi(\cO_S)$ and $h_SK_S=2\pi(S)-12$, hence $\chi(\cO_S)=\pi(S)-6$ and $\pi(S)\le10$. Thus the results in \cite[Theorem 0.1 and its proof in Section 9]{Ra} lead us to deal only with the following cases.
\begin{itemize}
\item An abelian surface.
\item A bielliptic surface.
\item A minimal surface $S$ of general type with $h_SK_S=6$, $K_S^2=3$, $p_g(S)=2$: in this case $\cO_S(h_S)\cong\cO_S(2K_S-A)$, where $A$ is a curve such that $A^2=-2$ and $AK_S=0$.
\item A minimal surface $S$ of general type with $h_SK_S=8$, $K_S^2=4$, $q(S)=0$, $p_g(S)=3$: in this case $\cO_S(h_S)\cong\cO_S(2K_S-A_1-A_2-A_3)$,  where the $A_i$'s are disjoint curves such that $A_i^2=-2$ and $A_iK_S=0$.
\end{itemize}

The existence of special Ulrich bundles on Abelian or bielliptic surfaces has been proved in \cite{Bea1,Bea3}. Such bundles can be taken $\mu$--stable: see \cite{Bea1} for abelian surfaces and  \cite{Cs6,Cs6erratum} for bielliptic surfaces. 

In the other cases $h^0\big(S,\cO_S(2K_S-h_S)\big)=1$, hence Theorem \ref{tMain} does not give any information on the  existence of special Ulrich bundles on such an $S$.
\end{example}

\begin{example}
\label{eCs}
Let $k\cong\bC$ and $S\subseteq\p4$ a non--degenerate surface of degree $d\ge 11$ with $\cO_S(h_S)$ non--special and $h_S^2> h_SK_S$. 

In \cite[Theorems 1 and 2]{I--M} and \cite[Proposition]{M--R} the authors classify non--special and non--degenerate surfaces $S\subseteq\p4$ which are not of general type. As a by--product they show that such surfaces have degree at most $10$. Thus $S$ is necessarily a surface of general type, hence $\chi(\cO_S)\ge1$.  Equality \eqref{RRGeneral}, the vanishings
\begin{equation}
\label{Vanishing}
h^0\big(S,\cO_S(K_S-h_S)\big)=0,\qquad h^1\big(S,\cO_S(K_S-h_S)\big)=h^1\big(S,\cO_S(h_S)\big)=0
\end{equation}
and the equality $h^0\big(S,\cO_S(h_S)\big)=5$ (see \cite{Se}) imply that
$$
h_SK_S=2\chi(\cO_S)+d-10\ge3:
$$
hence the condition $h_SK_S< h_S^2=d$ yields $\chi(\cO_S)\le4$. The double point formula  (see \cite{Ha2}, Example A.4.1.3) is
\begin{equation}
\label{Double}
K_S^2=\chi(\cO_S)+25+\frac{d(d-15)}2.
\end{equation}

The Hodge index theorem for the divisors $h_S$, $K_S$, Equality \eqref{Double} and the hypothesis $h_SK_S< d$ then yields $d\le 12$. 

Taking into account the bound $1\le\chi(\cO_S)\le4$, computing $K_S^2,h_SK_S$, applying again the Hodge index theorem for the divisors $h_S$, $K_S$ and Equality \eqref{Double} one easily checks that the case $d=12$ cannot occur. If $d=11$, the same argument yields the  $\chi(\cO_S)=4$, $K_S^2=7$, $h_SK_S=9$. If $S$ is not minimal, then $K_S=K_0+E$ where $K_0E=0$, $K_0^2\ge8$ and $h_SK_0\le 8$. The Hodge index theorem for the divisors $h_S$, $K_0$, then yields a contradiction. It follows that $S$ is minimal. If $q(S)\ge1$, then $p_g(S)\ge4$, hence we should have $K_S^2\ge8$, thanks to \cite[Th\'eor\ga eme 6.1]{De--Bea}, a contradiction. We deduce that $S$ is minimal and regular. 

We do not know if such a surface exists. Anyhow, if it exists, let $C\in\vert K_S\vert$: we have $\deg(C)=h_SK_S=9$ and $p_a(C)=K_S^2+1=8$. The cohomology of the exact sequence
\begin{equation}
\label{seqCurve}
0\longrightarrow\cO_S(-C)\longrightarrow\cO_S\longrightarrow\cO_C\longrightarrow0,
\end{equation}
tensored by $\cO_S(2K_S-h_S)$, Equalities \eqref{Vanishing} and the Riemann--Roch theorem on $C$ imply
\begin{align*}
h^0\big(S,\cO_S(2K_S-h_S)\big)&=h^0\big(C,\cO_C\otimes\cO_S(2K_S-h_S)\big)=\\
&=h^0\big(C,\cO_C(K_C)\otimes\cO_S(h_S-2K_S)\big)-2.
\end{align*}
The adjunction formula on $S$ returns  $\cO_C(K_C)\cong\cO_C\otimes\cO_S(2K_S)$, hence
$$
h^0\big(S,\cO_S(2K_S-h_S)\big)=h^0\big(C,\cO_C(h_C)\big)-2
$$
Thus $h^0\big(C,\cO_C(h_C)\big)\ge2$. If $h^0\big(C,\cO_C(h_C)\big)=2$, then the cohomology of Sequence \eqref{seqCurve} tensored by $\cO_S(h_S)$ would imply that $C$ is contained in three linearly independent hyperplanes of $\p4$, hence $C$ would be a line, a contradiction. We deduce that $3\le h^0\big(C,\cO_C(h_C)\big)$, hence $1\le h^0\big(S,\cO_S(2K_S-h_S)\big)$: again Theorem \ref{tMain} does not give any information on Ulrich bundles.
\end{example}

\bigskip
\noindent
Gianfranco Casnati,\\
Dipartimento di Scienze Matematiche, Politecnico di Torino,\\
c.so Duca degli Abruzzi 24,\\
10129 Torino, Italy\\
e-mail: {\tt gianfranco.casnati@polito.it}

\end{document}